\documentclass{amsart}
\usepackage{amsfonts}

\usepackage{graphicx}
\usepackage{amsmath}
\DeclareMathOperator{\arctanh}{arctanh}
\newtheorem{theorem}{Theorem}

\newtheorem{example}[theorem]{Example}
\newtheorem{lemma}[theorem]{Lemma}

\newtheorem{remark}[theorem]{Remark}
\newcommand{\R}{\mathbb{R}}
\newcommand{\U}{\mathbb{H}}
\newcommand{\INT}{\int_{-\infty}^{\infty}}

\begin{document}

\title{Krein condition and the Hilbert transform}
\author{Marcos L\'{o}pez-Garc\'{i}a}

%\thanks{
%The author was partially supported by project PAPIIT IN100919 of DGAPA-UNAM, and project A1-S-17475 of Conacyt, M\'exico.}
%\centerline{\today}
\email{marcos.lopez@im.unam.mx}
\address{
Instituto de Matem\'{a}ticas-Unidad Cuernavaca \\
   Universidad Nacional Aut\'{o}noma de M\'{e}xico\\
   Apdo. Postal 273-3, Cuernavaca Mor. CP 62251, M\'exico}

\maketitle

\begin{abstract}
Krein condition have been used as a qualitative result to show the M-indeterminacy of some kind of densities. In this work we use results from the theory of the Hilbert transform to construct families of densities having all the same finite moment sequence as a density $f$ with finite logarithmic integral. Actually, our approach explicitly gives Stieltjes classes with center at $f$ and perturbations involving the Hilbert transform of $\ln f.$ We consider densities supported on the whole real line or the positive half line.
\end{abstract}

\section{Introduction}
Let $F$ be a distribution supported on $I=\mathbb{R}^{+}$ or $\mathbb{R}$ such that 
$$\int_I x^n dF(x) < \infty \quad \text{for all } n\geq 1.$$
Under this assumption we say that $F$ has a finite moment sequence on $I$. A distribution $F$ with finite moment sequence on $I$ is called $M$-indeterminate if there are other distributions supported on $I$ having the same moments as $F$. \\

In 1945 Krein proved that if $F$ is an absolutely continuous distribution on $\R$ with finite moment sequence whose density $f$ has finite logarithmic integral, i.e. 
\begin{equation}\label{logint}
\int_{\mathbb{R}}-\frac{\log f\left( x\right) }{1+x^{2}}dx<\infty ,
\end{equation}
then $F$ is $M$-indeterminate. This is the so-called Krein criterion.\\

About the Krein criterion, in \cite{Ostro} the authors say that it ``is a qualitative result; there is no indication of how to write other distributions with the same
moments as $F".$ In \cite[Theorem 1]{Lin} the author used the theory of the Hardy space on the upper half plane $H^1$ to get a simple proof of the Krein criterion. In fact, if $f$ is a density satisfying the Krein condition (\ref{logint}), the author proved the existence of a density $g$ having the same moment sequence as $f$. In this work we go a step further, we combine the ideas in the proof of Theorem 1 in \cite{Lin} with some results of the Hilbert transform and the space $H^1$, to obtain an explicit description of the latter density $g$.\\ 

Actually, in this setting, we get a family of densities having the same moment sequence as $f$. To do this, we consider a construction introduced in \cite{Stoyanov1} to exhibit some
densities with the same moment sequence. \\

Let $f$ be a density with finite moment sequence on $I$. Assume that there exists a bounded measurable function $h$ with $\sup_{x\in I}\left|
h\left( x\right) \right| \leq 1,$ such that 
$$\int_{I}x^{n}f\left( x\right) h\left( x\right) dx=0 \quad\text{ for all }n\geq 0$$
and the function $fh$ is not identically zero, then the Stieltjes class $S_{I}\left( f,h\right) $ with center at $f $ and perturbation $h$ is given by
\begin{equation*}
S_{I}\left( f,h\right) =\left\{ f\left(x\right) \left[ 1+\varepsilon h\left( x\right) \right] :x\in I,\text{ }\varepsilon \in \left[ -1,1\right] \right\} .
\end{equation*}
Clearly, $S_{I}\left( f,h\right) $ is an infinite familiy of densities all having the same moment sequence as $f$. \\

Thus, our main results can be written in terms of Stieltjes classes involving the Hilbert transform of $\ln f$.
\begin{theorem}\label{ham}
Let $f$ be a density on $\mathbb{R}$ with finite moment
sequence. If $f$ has finite logarithmic integral, then $S_{\mathbb{R}}\left( f,\cos( \mathcal{H}\ln f) \right) $ and $S_{%
\mathbb{R}}\left( f,\sin ( \mathcal{H}\ln f) \right) $ are
Stieltjes classes, where
$$\mathcal{H}u(t)=\frac{1}{\pi} P  \INT \left(\frac{1}{t-x}+\frac{x}{1+x^2} \right)u(x)dx, \quad t\in \mathbb{R}.$$
\end{theorem}

When $I=\R^+$ we have a similar result.
\begin{theorem}\label{sti}
Let $f$ be a density on $\mathbb{R}^+$ with finite moment
sequence. If $f$ satisfies the condition
\begin{equation}\label{parasti}
\int_0^\infty -\frac{\ln f(x^2)}{1+x^2}dx <\infty,
\end{equation}
then $S_{\mathbb{R}^+}( f,\sin (\widetilde {\mathcal{H}}\ln f) ) $ 
is a Stieltjes class, where
$$\mathcal{H}_eu(t)=\frac{2t^{1/2}}{\pi}P \int_0^\infty \frac{u(x^2)}{t-x^2}dx,\quad t>0.$$
\end{theorem}

This work is organized as follows. In the next section we give some facts about the Hilbert transform and compute the Hilbert transform of two important cases. In the last section we prove the results and analyze two examples to show the usefulness of our approach.

\section{Preliminaries}\label{pre}
The following results can be found in \cite[pages 60-65]{Koosis}. Suppose that the function $u:\R \rightarrow \R$ satisfies
\begin{equation}\label{krein}
\int_{-\infty}^{\infty}\frac{|u(t)|}{1+t^2}dt < \infty.
\end{equation}
Hence the following integral
\begin{eqnarray*}
U(z)+i\widetilde{U}(z)&:=&\frac{i}{\pi} \int_{-\infty}^{\infty} \left(\frac{1}{z-t} +\frac{t}{1+t^2}\right)u(t)dt\\
               &=&\frac{1}{\pi} \INT \frac{\Im z}{|z-t|^2}u(t)dt+i\frac{1}{\pi}\INT \left(\frac{\Re z-t}{|z-t|^2}+ \frac{t}{1+t^2}  \right)u(t)dt
\end{eqnarray*}
converges absolutely on $\mathbb{H}:=\{z\in \mathbb{C}: \Im z>0\}$ and defines an analytic function on $\U$. Notice that $U$ is the Poisson integral of $u$ and is the unique harmonic extension of $u$ to $\U$. Moreover, $\widetilde{U}$ is the unique conjugate harmonic function of $U$ such that $\widetilde{U}(i)=0$.\\

It is known the existence of the non-tangential limits of $U$ and $\widetilde{U}$ at almost $t\in \R$; the non-tangential limit of $U$ is $u$, and the non-tangential limit of $\widetilde{U}$ is called the Hilbert transform of $u$ and is denoted by $\mathcal{H}u$. The Hilbert transform of $u$ can be written as the principal value of a singular integral:
\begin{equation*}
\mathcal{H}u(t)=\frac{1}{\pi}  \lim_{\varepsilon \rightarrow 0}\int_{|x-t|>\varepsilon}\left(\frac{1}{t-x}+\frac{x}{1+x^2} \right)u(x)dx, \quad a.e. \,\, t\in \R. 
\end{equation*}
%\begin{eqnarray*}
%\mathcal{H}u(t)&=&\frac{1}{\pi} P  \INT \left(\frac{1}{t-x}+\frac{x}{1+x^2} \right)u(x)dx \\
%&=& \frac{1}{\pi}  \lim_{\varepsilon \rightarrow 0}\int_{|x-t|>\varepsilon}\left(\frac{1}{t-x}+\frac{x}{1+x^2} \right)u(x)dx. 
%\end{eqnarray*}

\begin{remark}\label{even}
a) If $u$ is an even function satisfying (\ref{krein}) then
$$\mathcal{H}u(t)= \frac{2t}{\pi}P \int_0^\infty \frac{u(x)}{t^2-x^2}dx, \quad a.e. \,\, t\in \R.$$
In particular, $\mathcal{H}u$ is an odd function on $\R$. We also can see that $\mathcal{H}c=0$ where $c$ is a constant function.\\
b) Let $u:\R^+\rightarrow \R$ be such that $u\in L^1(dt/(1+t^2))$, then $\mathcal{H}_eu(t)=\mathcal{H}u^*(t^{1/2}),$ $t>0$, where $u^*(x)=u(x^2),$ $x\neq 0$.
\end{remark}

\begin{lemma}\label{hilpower}
Let $0<|\mu |<1$. The function $h_\mu(x)=|x|^\mu$ satisfies $(\ref{krein})$ and
\begin{equation}\label{Mu}
\mathcal{H}h_\mu(t)=-\tan(\mu\pi/2) \text{sgn}( t) |t|^\mu,\quad t\neq 0. 
\end{equation}
In particular, $\mathcal{H}_e(x^{\mu})(t)=-\tan(\mu\pi)t^\mu$, $t>0$.
\end{lemma}

\begin{proof} Let $\mu \in (-1,0)$. From \cite[Table 1.2, page 464]{King} and Remark 3 we have
\begin{eqnarray*}
-\tan(\mu\pi/2) \text{sgn}(t) |t|^{\mu}&=&\frac{1}{\pi}P\INT\frac{|x|^{\mu}}{t-x}dx\\
&=&\frac{2t}{\pi}P\int_0^\infty\frac{x^{\mu}}{t^2-x^2}dx=\mathcal{H}h_{\mu}(t), \quad t\neq 0.
\end{eqnarray*}

Let $\mu \in (0,1)$. From Remark \ref{even} and the previous case we obtain
\begin{eqnarray*}
\mathcal{H}h_\mu(t)&=&\frac{2t}{\pi}P \int_0^\infty \frac{x^{\mu}}{t^2-x^2}dx\\
&=&-\frac{2}{\pi t}P \int_0^\infty \frac{x^{-\mu}}{t^{-2}-x^2}dx\\
&=&-\tan(\mu\pi/2)\text{sgn}(t) |t|^\mu, \quad t \neq 0.
\end{eqnarray*}
\end{proof}

\begin{lemma}\label{hilon}
$\mathcal{H}(\ln|x|)(t)= -\pi/2$ \text{sgn}$(t)$ for  $t\neq 0$. In particular, $\mathcal{H}_e(\ln x)\equiv -\pi.$
\end{lemma}
\begin{proof}
From Remark \ref{even} we have
$$\mathcal{H}(\ln|x|)(t)=\frac{2t}{\pi}P \int_0^\infty \frac{\ln x}{t^2-x^2}dx,$$
and it is sufficient to consider $t>0$. Now, from the identity
$$ \int x^a \ln xdx=\frac{x^{a+1} \ln x}{a+1}-\frac{x^{a+1}}{(a+1)^2},$$
we get for $\varepsilon>0$ small enough that
\begin{eqnarray*}
\frac{1}{t^2}\int_0^{t-\varepsilon}\frac{\ln x}{1-(x/t)^2}dx&=&\sum_{n=0}^{\infty}\frac{1}{t^{2n+2}}\int_0^{t-\varepsilon}x^{2n}\ln xdx\\
&=&\left. \sum_{n=0}^{\infty}\frac{1}{t^{2n+2}}\left[\frac{x^{2n+1} \ln x}{2n+1}-\frac{x^{2n+1}}{(2n+1)^2}\right|_{x=0}^{x=t-\varepsilon}\right]\\
&=& \sum_{n=0}^{\infty}\frac{(t-\varepsilon)^{2n+1} \ln(t-\varepsilon)}{(2n+1)t^{2n+2}}-\frac{(t-\varepsilon)^{2n+1}}{(2n+1)^2t^{2n+2}}\\
&=& t^{-1}\ln(t-\varepsilon)\arctanh((t-\varepsilon)/t)-\sum_{n=0}^{\infty}\frac{(t-\varepsilon)^{2n+1}}{(2n+1)^2t^{2n+2}}\\
\end{eqnarray*}
Similarly, we get
\begin{eqnarray*}
-\int^\infty_{|t|+\varepsilon}\frac{1}{x^2}\frac{\ln x}{1-(t/x)^2}dx&=& \sum_{n=0}^{\infty}\frac{-t^{2n} \ln(t+\varepsilon)}{(2n+1)(t+\varepsilon)^{2n+1}}-\frac{t^{2n}}{(2n+1)^2(t+\varepsilon)^{2n+1}},\\
&=&-t^{-1} \ln(t+\varepsilon)\arctanh(t/(t+\varepsilon))-\sum_{n=0}^{\infty}\frac{t^{2n}/(t+\varepsilon)^{2n+1}}{(2n+1)^2},
\end{eqnarray*}
then we use that $\arctanh(x)=2^{-1}\ln\frac{1+x}{1-x} ,$ $|x|<1$, and apply the Weierstrass M-test considering $\varepsilon\in[0,\varepsilon_0)$ with $\varepsilon_0$ small enough, to obtain
\begin{eqnarray*}
\mathcal{H}(\ln|x|)(t)&=&\frac{1}{\pi}\lim_{\varepsilon\rightarrow 0^+} \ln\frac{(t-\varepsilon)(2t-\varepsilon)}{(t+\varepsilon)(2t+\varepsilon)} \\
&&-\frac{2}{\pi}\lim_{\varepsilon\rightarrow 0^+}\sum_{n=0}^{\infty}\frac{1}{(2n+1)^2}\left( \left((t-\varepsilon)/t\right)^{2n+1}+(t/(t+\varepsilon))^{2n+1}\right)\\
&=&-\frac{4}{\pi}\sum_{n=0}^{\infty}\frac{1}{(2n+1)^2}=-\frac{\pi}{2},\quad t>0.
\end{eqnarray*}
\end{proof}
%\frac{(|t|-\varepsilon)^{2n+1} }{t^{2n+1}}    %%%%-\frac{t^{2n+1} }{(|t|+\varepsilon)^{2n+1}}\ln(|t|+\varepsilon)
\section{Proof of the results}
\begin{proof}[\textbf{Proof of Theorem \ref{ham}}.]
Since $\ln f\leq f$, condition (\ref{logint})\ is equivalent to $\ln f\in L^{1}(
dt/( 1+t^{2}))$, so we can set $u=\ln f$ and proceed as at the beginning of Section \ref{pre}: consider the holomorphic function $F(z)=U(z)+i\widetilde{U}(z)$ on $\U$, where $U$ is the harmonic
extension of $u$\ to $\U$, and $\widetilde{U}$ is the unique conjugate harmonic function of $U$ satisfying $\widetilde{U}(i)=0.$\\

Now we introduce the function $G\in hol(\U)$ given by
\begin{equation*}
G\left( z\right) =\exp \circ F (z),\quad z\in \U. %\left( u\left( z\right) +iv\left( z\right) \right) 
\end{equation*}

By using the Jensen's inequality we get that 
\begin{equation*}
\left| G\left( z\right) \right| =\exp U\left( z\right) \leq \frac{1}{\pi } \INT \frac{\Im z}{|z-t|^2} f\left( t\right) dt, \quad z\in\U,
\end{equation*}
therefore 
\begin{equation*}
\INT\left| G\left( x,y\right) \right| dx\leq \INT f\left( t\right) dt=1\text{ for all }y>0.
\end{equation*}
Thus $G \in H^{1}$ and Theorem 3.1 in \cite[page 55]{Garnett}  implies that there exists a function $g\in L^{1}( \R) $ such that 
\begin{equation*}
g( x) =\lim_{y\rightarrow 0^+}G\left( x,y\right) ,\text{\ }a.e.\text{ }x\in \R .
\end{equation*}
By the other hand, we have
\begin{eqnarray*}
\lim_{y\rightarrow 0^+} G\left( x,y\right) &=&\exp\left(\lim_{y\rightarrow 0^+}U(x,y)\right) \exp\left(i\lim_{y\rightarrow 0^+}\widetilde{U}(x,y)\right)\\
                                                              & =& f\left( x\right) \exp\left(i(\mathcal{H}\ln f )(x)\right),\quad a.e.\text{ }x\in \mathbb{R},
\end{eqnarray*}
therefore 
\begin{equation} \label{fuente}
g(x)=f\left( x\right) \exp\left(i(\mathcal{H}\ln f )(x)\right) \quad \text{a.e on } \R.
\end{equation}
In particular, we notice that $g$ has finite moments of all nonnegative orders.\\

By Lemma 3.7 in \cite[page 59]{Garnett} we have
$$\INT g(x)e^{itx}dx =0\quad \text{ for all } t\geq 0, $$
which implies that 
$$\INT (ix)^k g(x)e^{itx}dx =0\quad \text{ for all } t\geq 0. $$
We set $t=0$ to get
\begin{equation}\label{vanish}
\INT x^k \Re g(x) dx= \INT x^k \Im g(x) dx= 0\quad \text{ for all } k\geq 0.
\end{equation}
Since $f$ is a density and $|g|=f$ a.e. on $\R$, it follows that at least one of the functions $\Re g, \Im g$ is a nonzero function. From (\ref{fuente}) and (\ref{vanish}) we get that $\cos(\mathcal{H}\ln f)$ and $\sin(\mathcal{H}\ln f)$ are perturbations for Stieltjes classes with center at $f$.
\end{proof}

\begin{example} Odd powers of the normal distribution. Let $X$ be a random variable with $X \sim N(0,\frac{1}{2})$, then $X^{2n+1}$, $n\geq 1$, has the density
$$f_{n}(x):=\frac{1}{(2n+1)\sqrt{\pi}}|x|^{-2n/(2n+1)}\exp(-|x|^{2/(2n+1)}), \quad x\in \R.$$ 
Clearly $f_n$ has a finite moment sequence. In \cite{Stoyanov4} was shown that $f_{n}$ has finite logarithmic integral for all $n \geq 1$. Lemmas \ref{hilpower} and \ref{hilon} imply that
$$\mathcal{H}\ln f_n(t)=\text{sgn}(t)\left( \pi n/(2n+1)+\tan\left( \pi /(2n+1)\right)|t|^{2/(2n+1)} \right), \, t\neq 0,$$
therefore $S_{\mathbb{R}}\left( f_n,h_c^n \right) $ and $S_{\mathbb{R}}\left( f_n,h_s^n) \right) $ are Stieltjes classes for all $n\geq 1$, where
\begin{eqnarray*}
h_c^n(t)&=&\cos(\mathcal{H}\ln f_n(t))=\cos\left( \pi n/(2n+1)+\tan\left( \pi /(2n+1)\right)|t|^{2/(2n+1)} \right)\\
              &=& \sin(\pi/(4n+2))\cos(\beta_n|t|^{2/(2n+1)})-  \cos(\pi/(4n+2))\sin(\beta_n|t|^{2/(2n+1)}), %\, t\neq 0,
\end{eqnarray*}
and
\begin{eqnarray*}
h_s^n(t)&=&\sin(\mathcal{H}\ln f_n(t))=\text{sgn}(t)\sin\left( \pi n/(2n+1)+\tan\left( \pi /(2n+1)\right)|t|^{2/(2n+1)} \right)\\
             &=&\text{sgn}(t) \left(\sin(\pi/(4n+2))\sin(\beta_n|t|^{2/(2n+1)})+  \cos(\pi/(4n+2))\cos(\beta_n|t|^{2/(2n+1)})\right),
\end{eqnarray*}
with $\beta_n=\tan\left( \pi /(2n+1)\right)$, $t\neq 0$. The perturbation $h_c^n$ was obtained for the first time in \cite{Berg}. As far as we know $h_s$ is a new perturbation, we can proceed as in \cite{Berg} to verify that $f_n h_s$ has vanishing moments.
\end{example}

\begin{proof}[\textbf{Proof of Theorem \ref{sti}}]
We set $f^*(x)=|x|f(x^2)$, $x \neq 0$. Clearly $f^*$ is a density on $\R$ and verifies
$$\INT-\frac{\ln f^*\left( x\right) }{1+x^{2}}dx=-2\int_0^{\infty}\frac{\ln x }{1+x^{2}}dx-2\int_0^{\infty}\frac{\ln f\left( x^2\right) }{1+x^{2}}dx<\infty .$$
The hypothesis about $f$ imply that $f^*$ has a finite moment sequence. Actually, since $f^*$ is an even function, the moments of odd order of $f^*$ vanish.\\

Hence $f^*$ satisfies the hypothesis in Theorem \ref{ham} and we can proceed as in (\ref{vanish}) to get
\begin{equation*}
\INT x^{2k} f^*(x) \cos\left( \mathcal{H}\ln f^* (x)\right) dx= \INT x^{2k+1} f^*(x) \sin\left( \mathcal{H}\ln f^* (x)\right) dx= 0
\end{equation*}
for all $k\geq 0$. Remark \ref{even} implies that
\begin{eqnarray*}
\INT x^{2k} f^*(x) \cos\left( \mathcal{H}\ln f^* (x)\right) dx&=&2\int_0^\infty x^{2k+1}f(x^2)\cos\left( \mathcal{H}\ln f^* (x)\right) dx \\
&=& \int_0^\infty x^k f(x) \cos\left( \mathcal{H}\ln f^* (x^{1/2})\right) dx=0
\end{eqnarray*}
for all $k\geq 0$. Similarly,
$$\INT x^{2k+1} f^*(x) \sin\left( \mathcal{H}\ln f^* (x)\right) dx=\int_0^\infty x^{k+1/2} f(x) \sin\left( \mathcal{H}\ln f^* (x^{1/2})\right) dx=0$$
for all $k\geq 0$. Unfortunaly, we notice the function $x^{1/2}\sin\left( \mathcal{H}\ln f^* (x^{1/2})\right)$ is not bounded on $\R^+$.\\

Remark \ref{even} and Lemma \ref{hilon} imply that
\begin{eqnarray*}
\mathcal{H}\ln f^* (t)&=&\mathcal{H}(\ln|x|)(t)+\mathcal{H}\left(\ln f(x^2)\right)(t)\\
&=& -\frac{\pi}{2} \text{sgn}(t)+\frac{2t}{\pi}P \int_0^\infty \frac{\ln f(x^2)}{t^2-x^2}dx.
\end{eqnarray*}
Finally, for $t>0$ we have
$$\mathcal{H}\ln f^* (t^{1/2})=-\frac{\pi}{2}+\frac{2t^{1/2}}{\pi}P \int_0^\infty \frac{\ln f(x^2)}{t-x^2}dx,$$
therefore $\cos\left( \mathcal{H}\ln f^*(t^{1/2})\right)=\sin(\mathcal{H}_e\ln f(t))$ is a perturbation for the Stieltjes class with center at $f$.
\end{proof} 

\begin{example}
Let $X\sim N(0,\frac{1}{2})$. For $r>0$ the random variable $|X|^r$ has a density supported on $\R^+$ given by
$$f_r(x):=\frac{2}{r\sqrt{\pi}}x^{1/r-1}\exp(-x^{2/r}), \quad x>0.$$
Clearly $f_r$ has a finite moment sequence. In \cite{Stoyanov4} was shown that $f_r$ satisfies the condition (\ref{parasti}) iff $r>4$.
In this case, Lemmas \ref{hilpower} and \ref{hilon} imply that
\begin{eqnarray*}
\mathcal{H}_e\left(\ln f_r\right)(t)&=&\left(1/r-1\right)\mathcal{H}_e (\ln x)(t)-\mathcal{H}_e\left(x^{2/r}\right)(t)\\
&=& (1-1/r)\pi+\tan(2\pi/r)t^{2/r}, \quad t> 0.
\end{eqnarray*}
Therefore $S_{\mathbb{R}^+}\left( f_r,h_r \right) $ is a Stieltjes class for all $r>4$ where
\begin{eqnarray*}
h_r(t)&=&\sin(\mathcal{H}_e\left(\ln f_r\right)(t))=\sin((1-1/r)\pi+\tan(2\pi/r)t^{2/r})\\
&=&  \sin(\pi/r)\cos(\tan(2\pi/r)t^{2/r})-\cos(\pi/r)\sin(\tan(2\pi/r) t^{2/r}),  \quad t>0.
\end{eqnarray*}
This perturbation was also obtained in \cite{Berg}.
\end{example}

\textbf{Conclusion} Krein condition is no longer just a qualitative result to show the $M$-indeterminacy of a density $f$ but provides families of densities having all the same moment sequence as $f$.

\bibliography{biblio_Stieltjes}{}
\bibliographystyle{plain}
\end{document}